\documentclass[12pt]{amsart}%
\usepackage{amssymb,amsthm, times}

\usepackage{delarray,verbatim}
\usepackage{graphicx}
\usepackage{color}
\usepackage[T1]{fontenc}
\usepackage{epsfig}

\usepackage{bm}

\usepackage{ifpdf}
\usepackage{graphics}
\usepackage{amsmath}
\usepackage{amsfonts}
\usepackage{amssymb}%
\usepackage[hmargin=0.8in,height=8.6in]{geometry}
\providecommand{\U}[1]{\protect\rule{.1in}{.1in}}

\newtheorem{theorem}{Theorem}

\newtheorem{proposition}{Proposition}

\newtheorem{lemma}{Lemma}
\newtheorem*{lemma*}{Lemma}

\newcommand{\E}{\mathbb{E}}

\newcommand{\R}{\mathbb{R}}

\def\R{\mathbb{R}}

\def\R{\mathbb{R}}

\def\E{\mathbb{E}}

\begin{document}
\title[A Random Matrix Approximation for the Non-Commutative Fractional Brownian Motion]{A Random Matrix Approximation for the Non-Commutative Fractional Brownian Motion}
\author[Juan Carlos Pardo]{Juan Carlos Pardo}
\address[Juan Carlos Pardo]{Department of Mathematics, Centro de Investigaci\'on en
Matem\'aticas, Apartado Postal 402, Guanajuato GTO 36000, Mexico.}
\email{jcpardo@cimat.mx}
\author[J. L. P\'erez]{Jos\'e-Luis P\'erez}
\address[J. L. P\'erez]{Department of Probability and Statistics, IIMAS-UNAM, Mexico
City, Mexico.}
\email{garmendia@sigma.iimas.unam.mx }
\author[Victor P\'{e}rez-Abreu]{Victor P\'{e}rez-Abreu}
\address[Victor P\'{e}rez-Abreu]{Department of Mathematics, Centro de Investigaci\'on
en Matem\'aticas, Apartado Postal 402, Guanajuato GTO 36000, Mexico.}
\email{pabreu@cimat.mx}
\thanks{This version: \today. }
\date{}
\maketitle

\begin{abstract}
A functional limit theorem for the empirical measure-valued process of
eigenvalues of a matrix fractional Brownian motion is obtained. It is shown
that the limiting measure-valued process is the non-commutative fractional
Brownian motion recently introduced by Nourdin and Taqqu \cite{NT}. Young and
Skorohod stochastic integral techniques and fractional calculus are the main
tools used.


\end{abstract}

\section{Introduction and main result}

Motivated by the fact that there is often a close correspondence between
classical probability and free probability, Nourdin and Taqqu
\cite{NT} recently introduced the non-commutative fractional Brownian motion (ncfBm).
It appears as the limiting process in a central limit theorem for long
range dependence time series in free probability, in analogy to the classical
probability case (see \cite{T}, for example). A ncfBm of Hurst parameter
$H\in(0,1)$ is a centered semicircular process $S^{H}=\left\{  S_{t}%
^{H}\right\}  _{t\geq0}$ in a non-commutative probability space $(\mathcal{A}%
,\varphi)$ with covariance function
\begin{equation}
\varphi(S_{t}^{H}S_{s}^{H})=\frac{1}{2}\left(  t^{2H}+s^{2H}-|t-s|^{2H}%
\right)  . \label{CFfBm}%
\end{equation}

The case $S^{1/2}$ is the well known free Brownian motion introduced in
\cite{B}. The ncfBm $S^{H}$ is the only standardized semicircular process
which is self-similar and has stationary increments. For the study of the
ncfBm and the required free probability framework, we refer to Section 2 in
\cite{NT} or Chapter 8 in \cite{No}. In the present paper, we will 
deal mainly with the law $(\mu_{t}^{H})_{t\geq0}$ of a ncfBm instead of the
non-commutative process.

Ever since the seminal paper by Voiculescu \cite{vu}, it has been well known that
free probability is a convenient framework for investigating the limits of
the spectral distributions of random matrices (see for instance Section 5.4 in 
Anderson, Guionnet and Zeitouni \cite{A}). On the functional asymptotic
behavior side, Biane \cite{B} proved that the free Brownian motion $S^{1/2}$
appears as the measure-valued process limit of $n\times n$ Hermitian matrix
Brownian motions with size $n$ going to infinity. Roughly speaking, this
result gives a realization of the free Brownian motion $S^{1/2}$ as the
spectral limit of well known matrix-valued processes.

On the other hand, for a fixed dimension $n,$ the matrix-valued fractional
Brownian motion was recently studied by Nualart and P\'{e}rez-Abreu \cite{AN}.
It was shown that its corresponding eigenvalue process is non-colliding almost
surely and a Skorohod stochastic differential equation governing this process
was established.

The main purpose of the present paper is to show that the ncfBm $S^{H}$ has a
realization as the measure-valued process limit of $n\times n$ matrix
fractional Brownian motions, as the size $n$ goes to infinity. This gives a
correspondence between classical fractional Brownian motion and
non-commutative fractional Brownian motion. Our method uses the Skorohod and
Young stochastic calculus for a multidimensional fractional Brownian motion as
well as the fractional calculus. It is important to note that our methodology does not apply  to the case $H=1/2$ of the
free Brownian motion.

More precisely, let us consider a family of independent fractional Brownian
motions starting from $0$ with Hurst parameter $H\in(1/2,1)$, $b=\{\{b_{ij}(t),t\geq0\},1\leq
i\leq j\leq n\}$, and define the symmetric matrix fractional Brownian motion
of dimension $n\times n$ by $B(t)$ by $B_{ij}(t)=b_{ij}(t)$ for $i<j$,
and $B_{ii}(t)=\sqrt{2}b_{ii}(t)$.

As we are interested in functional limit theorems for the eigenvalues of the
fractional Brownian motion, for $n\geq1$ we will consider the following
sequence of renormalized processes $\{B^{(n)}(t)\}_{t\geq0}$, given by
\[
B^{(n)}(t)=\frac{1}{\sqrt{n}}B(t),\qquad\text{for $t>0$.}%
\]
Following \cite{AN}, it is possible to apply the chain rule to the Young
integral to obtain the following equation for the eigenvalues of the process
$B^{(n)}$
\begin{equation}
\lambda_{i}^{(n)}(t)=\frac{1}{\sqrt{n}}\sum_{k\leq h}%
\int_{0}^{t}\frac{\partial\Phi_{i}^{(n)}}{\partial b_{kh}^{(n)}}(b(s))\circ
db_{kh}(s), \label{ito1}%
\end{equation}
for any $t>0$, $i=1,\dots,d$, and where $\Phi
_{i}^{(n)}=\lambda_{i}^{(n)}$. Observe 
\begin{equation}
\frac{\partial\Phi_{i}^{(n)}}{\partial b_{kh}^{(n)}}=2u_{ik}^{(n)}u_{ih}%
^{(n)}1_{\{k\not =h\}}+(u_{ik}^{(n)})^{2}1_{\{k=h\}} \label{eig}%
\end{equation}
where $u_{ik}^{(n)}$ denotes the $k$-th coordinate of the $i$-th eigenvector
of the matrix $B^{(n)}$.

The empirical measure-valued process which will be related to the functional
limit theorems is
\begin{equation}
\mu_{t}^{(n)}=\frac{1}{n}\sum_{j=1}^{n}\delta_{\lambda_{j}^{(n)}(t)}%
,\qquad\text{$t\geq0$,} \label{em}%
\end{equation}
where $\delta_{x}$ denotes the unit mass at $x$. From the celebrated Wigner theorem
in random matrix theory, one has that for each fixed $t>0$, $\mu_{t}^{(n)}$
converges a.s. to $\mu_{t}^{sc}$, the Wigner semicircle distribution of
parameter $t$:
\[
\mu_{t}^{sc}(dx)=\frac{1}{2\pi t}\sqrt{4t-x^{2}}1_{[-2\sqrt{t},2\sqrt{t}%
]}(x)dx,
\]
see for instance \cite{Me}, \cite{vu}, \cite{wi}.

The main result of this paper, stated in the framework of \cite{CL} and \cite{RS}, 
is the following functional limit theorem for
the empirical spectral measure-valued processes $\{(\mu_{t}^{(n)})_{t\geq
0}:n\geq1\}$ converging to the ncfBm. \ Let $\Pr(\mathbb{R})\ $be the space of
probability measures on $\mathbb{R}$ endowed with the topology of weak
convergence and let $C\left(  \mathbb{R}_{+},\Pr(\mathbb{R})\right)  \ $be the
space of continuous functions from $\mathbb{R}_{+}\ $into$\ \Pr(\mathbb{R}),$
endowed with the topology of uniform convergence on compact intervals of
$\mathbb{R}_{+}.$\ 

\smallskip

\begin{theorem}\label{limit2}
The family of measure-valued processes $\{(\mu_{t}^{(n)})_{t\geq0}%
:n\geq1\}$ converges weakly in $C(\mathbb{R}_{+},\mathrm{Pr}(\mathbb{R}))$ to the family $(\mu_{t})_{t\geq0}$ that corresponds to the law of a
non-commutative fractional Brownian motion of Hurst parameter $H\in(1/2,1)$
and covariance (\ref{CFfBm}). 
\end{theorem}

The case $H=1/2$ of the free Brownian motion is known, see for instance
\cite{CG}, \cite{CL}, \cite{AT0}, and \cite{RS}. The proof of Theorem \ref{limit2} is for $H\in(1/2,1)$
\ and it is done using results  about the Young stochastic integral as well as fine
estimations based on the fractional calculus.

The rest of this paper is organized as follows. In Section 2 we derive the stochastic
evolution of the empirical measure of the eigenvalues of the matrix fractional
Brownian motion. In Section 3 we prove that the family $\{(\mu_{t}^{n}%
)_{t\geq0}:n\geq1\}$ is tight in $C(\mathbb{R}_{+},\mathrm{Pr}(\mathbb{R}))$.
This is achieved by estimations of Young integrals by means of the fractional
calculus. In Section 4, we show that the weak limit $\{(\mu_{t})_{t\geq0}\}$, of the sequence of
measure-valued processes $\{(\mu_{t}^{(n)})_{t\geq0}:n\geq1\}$, satisfies a
measure-valued equation. In Section 5 we prove that the
deterministic process $\{(\mu_{t})_{t\geq0}\}$ corresponds to the law of a non-commutative fractional Brownian motion. 
For this we show, using results in \cite{vu}, that the process has  semicircular
finite-dimensional distributions, and covariance given by (\ref{CFfBm}).

\par For preliminaries on the stochastic calculus with respect to fractional Brownian
motion, we refer to \cite{nose}, \cite{No} and \cite{N}.

\section{The stochastic evolution of the empirical measure of the eigenvalues of a
matrix fractional Brownian motion}

As is usual, for a probability measure $\mu$ and a $\mu$-integrable
function $f$, we use the notation $\left\langle \mu,f\right\rangle =\int
f(x)\mu(\mathrm{d}x).$ Hence noting that the empirical measure $\{(\mu_t)_{t\geq0}\}$ is a point
measure, we have that for $f\in C_{b}^{2}$
\begin{equation}\label{nrem}
\langle\mu_{t}^{(n)},f\rangle=\frac{1}{n}\sum_{i=1}^{n}f(\lambda_{i}%
^{(n)}(t)).
\end{equation}
Therefore, applying the chain rule to the last equation,
\begin{equation}
\langle\mu_{t}^{(n)},f\rangle=\langle\mu_{0}^{(n)},f\rangle+\frac{1}{n}%
\sum_{i=1}^{n}\int_{0}^{t}f^{\prime}(\Phi^{n}_i(b(s)))\circ d\lambda_{i}%
^{(n)}(s),\qquad\text{$t\geq0$.} \label{emf}%
\end{equation}
In order to consider the evolution of the measure-valued process $\{\mu
_{t}^{(n)}:t\geq0\}$, we prove the following result.

\begin{lemma}
Let $\{\mu_{t}^{(n)}:t\geq0\}$ be the empirical measure-valued process of the
eigenvalues of the matrix fractional Brownian motion $(B^{(n)})_{t\geq0}$.  Then
for each $f\in C_{b}^{2}(\mathbb{R})$ and $t\geq0$ we have
\begin{align}
\langle\mu_{t}^{(n)},f\rangle &  =\langle\mu_{0}^{(n)},f\rangle+\frac
{1}{n^{3/2}}\sum_{i=1}^{n}\sum_{k\leq h}\int_{0}^{t}f^{\prime}(\Phi
^{n}_i(b(s)))\frac{\partial\Phi_{i}^{(n)}}{\partial b_{kh}^{(n)}}(b(s))\delta
b_{kh}(s)\nonumber\label{Ident1}\\
&  +H\int_{0}^{t}\int_{\mathbb{R}^{2}}\frac{f^{\prime}(x)-f^{\prime}(y)}%
{x-y}s^{2H-1}\mu_{s}^{(n)}(dx)\mu_{s}^{(n)}(dy)ds\nonumber\\
&  +\frac{H}{n^{2}}\sum_{i=1}^{n}\int_{0}^{t}f^{\prime\prime}(\Phi
^{n}_i(b(s)))s^{2H-1}ds.
\end{align}

\end{lemma}

\begin{proof}
First we note that using (\ref{emf}) and (\ref{ito1}) we obtain
\[
\langle\mu_{t}^{(n)},f\rangle=\langle\mu_{0}^{(n)},f\rangle+\frac{1}{n^{3/2}%
}\sum_{i=1}^{n}\sum_{k\leq h}\int_{0}^{t}f^{\prime}(\Phi^{n}_i(b(s)))\frac
{\partial\Phi_{i}^{(n)}}{\partial b_{kh}^{(n)}}(b(s))\circ db_{kh}(s).
\]
Now we will be interested in replacing the Young integrals by Skorohod
integrals in the above expression. To this end, we will prove that the
condition of Proposition 3 in \cite{EN} is satisfied. We will denote by $D^{kh}$ the Malliavin derivative with respect to $b_{kh}$, for each $1\leq k\leq h\leq n$.

First note that
\begin{align}
\int_{0}^{t}\int_{0}^{t}D^{kh}_r &  \bigg(f^{\prime}(\Phi^{n}_i(b(s)))\frac
{\partial\Phi_{i}^{(n)}}{\partial b_{kh}^{(n)}}(b(s))\bigg)|s-r|^{2H-2}%
drds\nonumber\\
&  =\frac{1}{2H-1}\int_{0}^{t}f^{\prime\prime}(\Phi^{n}_i(b(s)))\left(  \frac
{\partial\Phi_{i}^{(n)}}{\partial b_{kh}^{(n)}}(b(s))\right)  ^{2}%
s^{2H-1}ds\nonumber\\
&  +\frac{1}{2H-1}\int_{0}^{t}f^{\prime}(\Phi^{n}_i(b(s)))\frac{\partial^{2}%
\Phi_{i}^{(n)}}{\partial(b_{kh}^{(n)})^{2}}(b(s))s^{2H-1}ds.\nonumber
\end{align}
Therefore, using (\ref{eig}),
\[
\bigg|\frac{\partial\Phi_{i}^{(n)}}{\partial b_{kh}^{(n)}}(b(s))\bigg|\leq2
\]
and so
\[
\left\vert \int_{0}^{t}f^{\prime\prime}(\Phi^{n}_i(b(s)))\left(  \frac
{\partial\Phi_{i}^{(n)}}{\partial b_{kh}^{(n)}}(b(s))\right)  ^{2}%
s^{2H-1}ds\right\vert \leq\frac{4}{2H}\Vert f^{\prime\prime}\Vert_{\infty
}t^{2H}<\infty.
\]
On the other hand, using (5.6) in \cite{AN}, we obtain
\begin{align}
\mathbb{E}\left(  \left\vert \int_{0}^{t}f^{\prime}(\Phi^{n}_i(b(s)))\frac
{\partial^{2}\Phi_{i}^{(n)}}{\partial(b_{kh}^{(n)})^{2}}(b(s))s^{2H-1}%
ds\right\vert \right)   &  \leq\Vert f^{\prime}\Vert_{\infty}\mathbb{E}\left(
\int_{0}^{t}\left\vert \frac{\partial^{2}\Phi_{i}^{(n)}}{\partial(b_{kh}%
^{(n)})^{2}}(b(s))\right\vert s^{2H-1}ds\right)  \nonumber\\
&  =\Vert f^{\prime}\Vert_{\infty}\int_{0}^{t}\mathbb{E}\left(  \left\vert
\frac{\partial^{2}\Phi_{i}^{(n)}}{\partial(b_{kh}^{(n)})^{2}}(b(s))\right\vert
\right)  s^{2H-1}ds\nonumber\\
&  \leq C_{1}\int_{0}^{t}s^{H-1}ds=\frac{C_{1}}{H}t^{H}<\infty.\nonumber
\end{align}
Therefore, we can conclude that
\[
\left\vert \int_{0}^{t}f^{\prime}(\Phi^{n}_i(b(s)))\frac{\partial^{2}\Phi
_{i}^{(n)}}{\partial(b_{kh}^{(n)})^{2}}(b(s))s^{2H-1}ds\right\vert
<\infty\qquad\text{$\mathbb{P}$ a.s.}%
\]
So, putting the pieces together, we obtain that
\[
\int_{0}^{t}\int_{0}^{t}D^{kh}_r\bigg(f^{\prime}(\Phi^{n}_i(b(s)))\frac
{\partial\Phi_{i}^{(n)}}{\partial b_{kh}^{(n)}}(b(s))\bigg)|s-r|^{2H-2}%
drds<\infty\qquad\text{$\mathbb{P}$ a.s.}%
\]
Therefore, by Proposition 3 in \cite{EN} (see also Proposition 5.2.3 in
\cite{N}), we can express the Young integrals that appear in (\ref{ito1}) in
terms of Skorohod integrals. Therefore,
\begin{align}
\langle &  \mu_{t}^{(n)},f\rangle=\langle\mu_{0}^{(n)},f\rangle+\frac
{1}{n^{3/2}}\sum_{i=1}^{n}\sum_{k\leq h}\int_{0}^{t}f^{\prime}(\Phi
^{n}_i(b(s)))\frac{\partial\Phi_{i}^{(n)}}{\partial b_{kh}^{(n)}}(b(s))\delta
b_{kh}(s)\nonumber\\
&  +\frac{H(2H-1)}{n^{2}}\sum_{i=1}^{n}\sum_{k\leq h}\int_{0}^{t}\int_{0}%
^{t}D^{kh}_r\bigg(f^{\prime}(\Phi^{n}_i(b(s)))\frac{\partial\Phi_{i}^{(n)}%
}{\partial b_{kh}^{(n)}}(b(s))\bigg)|s-r|^{2H-2}drds\nonumber\\
&  =\langle\mu_{0}^{(n)},f\rangle+\frac{1}{n^{3/2}}\sum_{i=1}^{n}\sum_{k\leq
h}\int_{0}^{t}f^{\prime}(\Phi^{n}_i(b(s)))\frac{\partial\Phi_{i}^{(n)}}{\partial
b_{kh}^{(n)}}(b(s))\delta b_{kh}(s)\nonumber\\
&  +\frac{H}{n^{2}}\sum_{i=1}^{n}\sum_{k\leq h}\int_{0}^{t}f^{\prime\prime
}(\Phi^{n}_i(b(s)))\left(  \frac{\partial\Phi_{i}^{(n)}}{\partial b_{kh}^{(n)}%
}(b(s))\right)  ^{2}s^{2H-1}ds\nonumber\\
&  +\frac{H}{n^{2}}\sum_{i=1}^{n}\sum_{k\leq h}\int_{0}^{t}f^{\prime}(\Phi
^{n}_i(b(s)))\frac{\partial^{2}\Phi_{i}^{(n)}}{\partial(b_{kh}^{(n)})^{2}%
}(b(s))s^{2H-1}ds.\nonumber
\end{align}
On the other hand in p. 4280 of \cite{AN} we can find the following relation
\begin{equation}
\sum_{k\leq h}\left(  \frac{\partial\Phi_{i}^{(n)}}{\partial b_{kh}^{(n)}%
}(b(s))\right)  ^{2}=2.\label{om}%
\end{equation}
Hence, using (\ref{om}),
\begin{align}
\langle\mu_{t}^{(n)},f\rangle &  =\langle\mu_{0}^{(n)},f\rangle+\frac
{1}{n^{3/2}}\sum_{i=1}^{n}\sum_{k\leq h}\int_{0}^{t}f^{\prime}(\Phi
^{n}_i(b(s)))\frac{\partial\Phi_{i}^{(n)}}{\partial b_{kh}^{(n)}}(b(s))\delta
b_{kh}(s)\nonumber\\
&  +\frac{2H}{n^{2}}\sum_{i=1}^{n}\int_{0}^{t}f^{\prime\prime}(\Phi
^{n}_i(b(s)))s^{2H-1}ds+\frac{2H}{n^{2}}\sum_{i=1}^{n}\sum_{j\not =i}\int
_{0}^{t}\frac{f^{\prime}(\Phi^{n}_i(b(s)))}{\lambda_{i}^{(n)}(s)-\lambda
_{j}^{(n)}(s)}s^{2H-1}ds\nonumber\\
&  =\langle\mu_{0}^{(n)},f\rangle+\frac{1}{n^{3/2}}\sum_{i=1}^{n}\sum_{k\leq
h}\int_{0}^{t}f^{\prime}(\Phi^{n}_i(b(s)))\frac{\partial\Phi_{i}^{(n)}}{\partial
b_{kh}^{(n)}}(b(s))\delta b_{kh}(s)\nonumber\\
&  +H\int_{0}^{t}\int_{\mathbb{R}^{2}}\frac{f^{\prime}(x)-f^{\prime}(y)}%
{x-y}s^{2H-1}\mu_{s}^{(n)}(dx)\mu_{s}^{(n)}(dy)ds\nonumber\\
&  +\frac{H}{n^{2}}\sum_{i=1}^{n}\int_{0}^{t}f^{\prime\prime}(\Phi
^{n}_i(b(s)))s^{2H-1}ds.\notag
\end{align}
Here, in the third equality, we used the identity
\[
\sum_{k\leq h}\frac{\partial^{2}\Phi_{i}^{(n)}}{\partial(b_{kh}^{(n)})^{2}%
}(b(s))=2\sum_{j\not =i}\frac{1}{\lambda_{i}^{(n)}(s)-\lambda_{j}^{(n)}(s)}.
\]
(See, for instance, p. 4279 in \cite{AN}).
\end{proof}

\section{Tightness of the family of laws $\{\mu^{(n)}_{t}:t\geq0\}$}

In this section we will prove that the family of the laws of the measured-valued
processes $\{(\mu_{t}^{(n)})_{t\geq0}:n\geq1\}$ is tight in the space
$C(\mathbb{R}_{+},\text{Pr}(\mathbb{R}))$.

\begin{proposition}
The family of measures $\{(\mu_{t}^{(n)})_{t\geq0}:n\geq1\}$ is tight.
\end{proposition}

\begin{proof}
Using (\ref{nrem}) it is easy to see that for
every $0\leq t_{1}\leq t_{2}\leq T$, $n\geq1$ and $f\in\mathcal{C}^{2}_{b}$,
\begin{align}
\label{est0}\Big|\langle\mu^{(n)}_{t_{2}}, f\rangle-\langle\mu^{(n)}_{t_{1}%
},f\rangle\Big|  &  \leq\frac{1}{n}\sum_{i=1}^{n}\Big|f(\lambda_{i}%
^{(n)}(t_{2}))-f(\lambda_{i}^{(n)}(t_{1}))\Big|.
\end{align}
We will assume that the eigenvalues are ordered in the following way
\[
\lambda_1^{(n)}(t)\leq \lambda_2^{(n)}(t)\leq \dots \leq\lambda_n^{(n)}(t),
\]
for each $t\geq0$.
\par Hence using Lemma 2.1.19 in \cite{A} (the Hoffman-Weilandt inequality, see also \cite{HW}), and the fact that the eigenvalues do not collide for any $t>0$ a.s., we deduce 
\[
|\lambda_{i}^{(n)}(t_2)-\lambda_{i}^{(n)}%
(t_1)|^4\leq \left[\sum_{i=1}^n(\lambda^{(n)}_i(t_2)-\lambda^{(n)}_i(t_1))^2\right]^2\leq\left[\frac{1}{n}\sum_{i,j=1}^n\left(\frac{B_{ij}(t_2)}{\sqrt{n}}-\frac{B_{ij}(t_1)}{\sqrt{n}}\right)^2\right]^2,
\]
for each $i=1,\dots,n$.
\par Therefore using the fact that the entries of the matrix fractional Brownian motion $(B(t))_{t\geq0}$ are independent, we obtain that there exists a constant $C>0$ that does not depend on $n$ such that 
\[
\E\left(|\lambda_{i}^{(n)}(t_2)-\lambda_{i}^{(n)}%
(t_1)|^4\right)\leq C |t_1-t_2|^{4H},\qquad\text{ for all $i=1,\dots,n$.}
\]
\par Again, using that the function $f^{\prime}$ is bounded and
applying the Mean Value Theorem, we deduce
\[
|f(\lambda_{i}^{(n)}(r))-f(\lambda_{i}^{(n)}(s))|\leq
\|f^{\prime}\|_{\infty}|\lambda_{i}^{(n)}(r)-\lambda_{i}^{(n)}%
(s)|.
\]
Therefore using the
above estimate in (\ref{est0}) and Jensen's inequality, we obtain
\begin{align}
\label{est2}\mathbb{E}\Bigg(\Big|\langle\mu^{(n)}_{t_{2}}, f\rangle 
-\langle\mu^{(n)}_{t_{1}},f\rangle\Big|^{4}\Bigg)&\leq\|f^{\prime}\|_{\infty}^4\E\left[\left(\frac{1}{n}\sum_{i=1}^{n}\Big|\lambda_{i}
^{(n)}(t_{2})-\lambda_{i}^{(n)}(t_{1})\Big|\right)^4\right]\notag\\
&\leq \frac{\|f'\|_{\infty}^4}{n}\sum_{i=1}^n\E\left(\Big|\lambda_{i}
^{(n)}(t_{2})-\lambda_{i}^{(n)}(t_{1})\Big|^4\right)\notag\\
&\leq C\|f'\|_{\infty}^4 |t_2-t_1|^{4H}.\nonumber
\end{align}

Therefore, by a well known criterion (see \cite{EK}, Prop.
2.4), we have that the sequence of continuous real processes $\{(\langle
\mu_{t}^{(n)},f\rangle)_{t\geq0} :n\geq1\}$ is tight and consequently the
sequence of processes $\{(\mu_{t}^{(n)})_{t\geq0} :n\geq1\}$ is tight in the
space $\mathcal{C}(\mathbb{R}_{+},\mathrm{Pr}(\mathbb{R}))$.
\end{proof}

\section{Weak convergence of the empirical measure of eigenvalues}

In the previous section, we proved that the family of measures $\{(\mu
_{t}^{(n)})_{t\geq0}:n\geq1\}$ is tight in the space $\mathcal{C}%
(\mathbb{R}_{+},\mathrm{Pr}(\mathbb{R}))$. Now we will proceed to identify the
limit of any subsequence of the family. To this end we will first prove an
estimate for the $p$th moment of the repulsion force between the eigenvalues
of a matrix fractional Brownian motion, as the dimension goes to infinity.

\begin{lemma}
For each $p\in(1,2)$, $i=1,\dots,n-1$, and for $t\geq0$ we have that
\[
\mathbb{E}\left(  \frac{1}{|\lambda^{(n)}_{i}(t)-\lambda^{(n)}_{i+1}(t)|^{p}%
}\right)  =t^{-pH}O(n^{p-2}),\qquad\text{as $n\to\infty$,}
\]
uniformly with respect to $t$.
\end{lemma}

\begin{proof}
For $t\geq0$, let us consider the eigenvalues $\{\lambda^{(n)}(t)\}_{i=1}^{n}$
of the matrix $B^{(n)}(t)$. Using (5.6) in \cite{AN}
and (7.2.30) in \cite{Me}, we have that the joint
distribution of two consecutive eigenvalues is given by
\[
\mathbb{P}(\lambda_{i}^{(n)}(1)\in dx_{i},\lambda_{i+1}^{(n)}(1)\in
dx_{i+1})=n\frac{(n-2)!}{n!}\text{det}[K_{n1}(x_{i}\sqrt{n}%
,x_{i+1}\sqrt{n})]dx_{i}dx_{i+1}\notag,
\]
where
\[ K_{n1}(u_i,u_{i+1})=\left( \begin{array}{cc}
S_n(u_i,u_{i+1})+\alpha_n(u_i)& D_n(u_i,u_{i+1})  \\
J_n(u_i,u_{i+1}) &  S_n(u_{i+1},u_i)+\alpha_n(u_{i+1})
\end{array} \right)\]
with
\begin{align}
S_n(u_i,u_{i+1})&=\sum_{j=0}^{n-1}\varphi_j(u_i)\varphi(u_{i+1})+\left(\frac{n}{2}\right)^{1/2}\varphi_{n-1}(u_i)\int_{\R}\delta(u_{i+1}-t)\varphi_n(t)dt,\notag\\
D_n(u_i,u_{i+1})&=-\frac{\partial}{\partial u_{i+1}}S_n(u_i,u_{i+1}),\notag\\
I_n(u_i,u_{i+1})&=\int_{\R}\delta(u_i-t)S_n(t,u_{i+1})dt,\notag\\
J_n(u_i,u_{i+1})&=I_n(u_i,u_{i+1})-\delta(u_i-u_{i+1})+\beta(u_i)-\beta(u_{i+1}),\notag\\
\beta(u_i)&=\int_{\R}\delta(u_i-y)\alpha(y)dy,\notag\\
\delta(u_i)&=\frac{1}{2}\text{sign}(u_i),\notag
\end{align}
\[
\alpha_n(u_i)= \begin{cases} \varphi_{2m}(u_i)/\int_{\R}\varphi_{2m}(t)dt &\mbox{if } n=2m+1 \\
0 & \mbox{if } n=2m, \end{cases}
\]
and for each $j\in\mathbb{N}$
\[
\varphi_j(u_i)=(2^jj!\sqrt{\pi})^{-1/2}\exp(u_i^2/2)\left(-\frac{d}{du_i}\right)^j\exp(-u_i^2).
\]
By Proposition 5.4 in \cite{AN} we have that the process $\{(\lambda_1^{(n)}(t),\dots,\lambda_n^{(n)}(t))\}_{t\geq0}$ is $H$-self-similar, hence
\[
\mathbb{E}\Bigg(\frac{1}{|\lambda_{i}^{(n)}(t)-\lambda_{i+1}^{(n)}(t)|^{p}} 
\Bigg)=t^{-pH}\mathbb{E}\Bigg(\frac{1}{|\lambda_{i}^{(n)}(1)-\lambda_{i+1}^{(n)}(1)|^{p}} 
\Bigg).
\]
Therefore, if we consider the following expectation, we have that
\begin{align}
\mathbb{E}\Bigg(\frac{1}{|\lambda_{i}^{(n)}(t)-\lambda_{i+1}^{(n)}(t)|^{p}}  &
\Bigg)=t^{-pH}n\frac{(n-2)!}{n!}\int_{\mathbb{R}}\int_{\mathbb{R}}\frac
{1}{|x_{i}-x_{i+1}|^{p}}\text{det}[K_{n1}(x_{i}\sqrt{n},x_{i+1}\sqrt
{n})]dx_{i}dx_{i+1}\nonumber\\
&  =t^{-pH}\frac{2\pi^{1-p} n^{p}}{n(n-1)}\int_{\mathbb{R}}\int_{\mathbb{R}}\frac
{1}{|u_{i}-u_{i+1}|^{p}}\frac{\pi}{2n}\text{det}[K_{n1}(\pi u_{i}/\sqrt{n},\pi
u_{i+1}/\sqrt{n})]du_{i}du_{i+1}.\nonumber
\end{align}
On the other hand, using (7.2.41) in \cite{Me} (see also Theorem 3.9.22 in \cite{A}), we have that the joint
density of the eigenvalues satisfies for any bounded interval $I\subset\R$ 
\begin{equation}\label{cu}
\lim_{n\rightarrow\infty}\frac{\pi}{2n}\text{det}[K_{n1}(\pi u_{i}/\sqrt
{n},\pi u_{i+1}/\sqrt{n})]=K(u_{i},u_{i+1}),
\end{equation}
uniformly on $u_{i},u_{i+1}\in I$, where
\[
K(u_{i},u_{i+1})=1-\left[s^{2}(r)+\left(  \int_{r}^{\infty}s(t)dt\right)  \left(
\frac{d}{dr}s(r)\right)\right] ,
\]
with $s(r)=\sin(\pi r)/\pi r$ and $r=|u_{i}-u_{i+1}|$. Hence using the
estimate (7.2.44) in \cite{Me}, we note that
\[
\int_{\mathbb{R}}\int_{\mathbb{R}}\frac{1}{|u_{i}-u_{i+1}|^{p}}K(u_{i}%
,u_{i+1})du_{i}du_{i+1}<\infty.
\]
Now we note that using (\ref{cu}) and Scheffe's Theorem (see p. 215 in \cite{PB}) the following holds
\begin{align}
\lim_{n\to\infty}\int_{|u_i|+|u_{i}-u_{i+1}|>1}\frac
{1}{|u_{i}-u_{i+1}|^{p}}&\frac{\pi}{2n}\text{det}[K_{n1}(\pi u_{i}/\sqrt{n},\pi
u_{i+1}/\sqrt{n})]du_{i}du_{i+1}\notag\\
&=\int_{|u_i|+|u_{i}-u_{i+1}|>1}\frac{1}{|u_{i}-u_{i+1}|^{p}}K(u_{i},u_{i+1})du_{i}du_{i+1}.\notag
\end{align}
Using that (\ref{cu}) holds uniformly on $u_{i},u_{i+1}\in I$, for any bounded interval $I\subset\R$, then we can find $N\in\mathbb{N}$, large enough, 
such that
\[
\left|\frac{\pi}{2n}\text{det}[K_{n1}(\pi u_{i}/\sqrt
{n},\pi u_{i+1}/\sqrt{n})]-K(u_{i},u_{i+1})\right|<\varepsilon,\qquad\text{for $n\geq N$.}
\]
On the other hand, using polar coordinates, we obtain
\begin{align}
\int_{|u_i|+|u_{i}-u_{i+1}|\leq1}\frac{1}{|u_{i}-u_{i+1}|^{p}}du_{i}du_{i+1}&\leq \int_0^{\sqrt{5}}r^{1-p}\int_0^{2\pi}\frac{1}{|\cos\theta-\sin\theta|^{p}}d\theta dr\notag\\
&\leq \frac{5^{1-p/2}}{2-p}\left[\int_0^{2\pi}\frac{1}{|\cos\theta-\sin\theta|^{2}}d\theta +\int_0^{2\pi}d\theta \right]\notag\\
&=\frac{5^{1-p/2}}{2-p}\left(2\pi+\frac{1}{2}\left[\tan\left(\frac{t\pi}{4}\right)+\tan\left(\frac{\pi}{4}\right)\right]\right)<\infty.\notag
\end{align}
Therefore
\begin{align}
\int_{|u_i|+|u_{i}-u_{i+1}|\leq 1}\frac
{1}{|u_{i}-u_{i+1}|^{p}}&\left|\frac{\pi}{2n}\text{det}[K_{n1}(\pi u_{i}/\sqrt{n},\pi
u_{i+1}/\sqrt{n})]-K(u_{i},u_{i+1})\right|du_{i}du_{i+1}\notag\\
&\leq \varepsilon \int_{|u_i|+|u_{i}-u_{i+1}|\leq 1}\frac
{1}{|u_{i}-u_{i+1}|^{p}},\notag
\end{align}
which in turn implies 
\begin{align}
\lim_{n\to\infty}\int_{|u_i|+|u_{i}-u_{i+1}|\leq 1}\frac
{1}{|u_{i}-u_{i+1}|^{p}}&\frac{\pi}{2n}\text{det}[K_{n1}(\pi u_{i}/\sqrt{n},\pi
u_{i+1}/\sqrt{n})]du_{i}du_{i+1}\notag\\
&=\int_{|u_i|+|u_{i}-u_{i+1}|\leq 1}\frac{1}{|u_{i}-u_{i+1}|^{p}}K(u_{i},u_{i+1})du_{i}du_{i+1}.\notag
\end{align}
Hence we finally obtain 
\[
\lim_{n\rightarrow\infty}n^{2-p}\mathbb{E}\left(  \frac{1}{|\lambda_{i}%
^{(n)}(t)-\lambda_{i+1}^{(n)}(t)|^{p}}\right)  =\pi t^{-pH}\int_{\mathbb{R}}%
\int_{\mathbb{R}}\frac{1}{|u_{i}-u_{i+1}|^{p}}K(u_{i},u_{i+1})du_{i}du_{i+1}.
\]
\end{proof}

The previous lemma will allow us to prove the following result related to the
convergence of the multidimensional Skorohod integral that appears in
(\ref{Ident1}), which in turn will enable us to identify the limit of any
subsequence of the family of laws of processes $\{(\mu_{t}^{(n)})_{t\geq
0}:n\geq1\}$.

\begin{lemma}
For any $T>0$, any $f:\mathbb{R}\to\mathbb{R}$ such that $f^{\prime}$ and
$f^{\prime\prime}$ are bounded, and $p\in(1/H,2)$, we have that
\begin{equation}
\label{ca0}\lim_{n\to\infty}\frac{1}{n^{3/2}}
\bigg|\sum_{i=1}^{n}\sum_{k\leq h}\int_{0}^{t}f^{\prime}(\Phi_{i}^{(n)}%
(b(s)))\frac{\partial\Phi_{i}^{(n)}}{\partial b^{(n)}_{kh}}(b(s))\delta
b_{kh}(s)\bigg|=0,\qquad\text{$t\in[0,T]$,}
\end{equation}
in probability.
\end{lemma}

\begin{proof}
Let us use the following notation for the Skorohod integral with respect to
the multidimensional fractional Brownian motion $\{b(t), t\ge0\}$:
\[
\int_{0}^{t}g^{i,n}(b(s))\delta b(s):=\sum_{k\leq h}\int_{0}^{t}g^{i,n}%
_{kh}(b(s))\delta b_{kh}(s),
\]
where
\[
\displaystyle g^{i,n}_{kh}(b(s)):=f^{\prime}(\Phi^{(n)}_{i}(b(s)))\frac{\partial
\Phi_{i}^{(n)}}{\partial b^{(n)}_{kh}}(b(s)), \qquad\text{for each}\quad
i=1,\dots,n.
\]

The following $\mathbb{L}_{p}$ estimates will be very useful in what follows.
For any $p\geq1/H$ and $i=1,\dots,n$,
\begin{align}\label{ISS2}
\bigg|\int_{0}^{T} g^{i,n}(b(s)) \delta b(s)\bigg|^{p}&\leq c_{p,T,H}
\left(  \|\mathbb{E}(g^{i,n}(b))\|^{p}_{\mathbb{L}^{1/H}([0,T])}+\mathbb{E}%
\|Dg^{i,n}(b)\|^{p}_{\mathbb{L}^{1/H}([0,T]^2)}\right)\notag\\
&=c_{p,T,H}\left[  \int_{0}^{T}\Vert\mathbb{E}%
(g^{i,n}(b(s)))\Vert^{p}ds+\mathbb{E}\left(  \int_{0}^{T}\left(  \int_{0}%
^{T}\Vert D_{s}g^{i,n}(b(r))\Vert^{\frac{1}{H}}ds\right)  ^{pH}dr\right)
\right] 
\end{align}
where $c_{p,T,H}$ is a positive constant depending on $p$, $H$, and $T$.

This last result is a consequence of Meyer's inequalities: it appears for the
one dimensional case in (5.40) of \cite{N} and can be extended to the
multidimensional case as in the proof of Proposition 3.5 of \cite{NP}.

Now, we proceed to estimate each of the two integrals in the right hand side
of (\ref{ISS2}). Recalling the definition of $g$ and using 
(\ref{om}), it is clear, by Jensen's inequality, that
\begin{align}
\Vert\mathbb{E}(g^{i,n}(b(s)))\Vert=\left[  \sum_{k\leq h}\left(  \mathbb{E}%
\left\{  f^{\prime}(\Phi_{i}^{(n)}(b(s)))\frac{\partial\Phi_{i}^{(n)}}{\partial
b_{kh}^{(n)}}(b(s))\right\}  \right)  ^{2}\right]  ^{1/2} &  \leq\Vert
f^{\prime}\Vert_{\infty}\left[  \sum_{k\leq h}\mathbb{E}\left(  \frac
{\partial\Phi_{i}^{(n)}}{\partial b_{kh}^{(n)}}(b(s))\right)  ^{2}\right]
^{1/2}\nonumber\\
&  =2^{1/2}\Vert f^{\prime}\Vert_{\infty}.\nonumber
\end{align}
Therefore, by Jensen's inequality, we get
\begin{equation}
\int_{0}^{T}\big\|\mathbb{E}(g^{i,n}(b(s)))\big\|^{p}ds\leq2^{p/2}\Vert f^{\prime}%
\Vert_{\infty}^{p}T.\notag
\end{equation}
For the second integral in the right hand side of (\ref{ISS2}), we first
compute an upper bound for the norm of the Malliavin derivative of $g$:
\begin{align}
\Vert D_{s}g^{i,n}(b(r))\Vert &  \leq\sum_{k\leq h}|(D_{s}^{kh}g_{kh}%
^{i,n}(b(r)))|\nonumber\label{2d1}\\
&  =\sum_{k\leq h}\left(  \frac{|f^{\prime\prime}(\Phi^{(n)}_i(b(r)))|}{\sqrt{n}}\left(  \frac
{\partial\Phi_{i}^{(n)}}{\partial b_{kh}^{(n)}}(b(r))\right)  ^{2}+\frac{|f^{\prime
}(\Phi^{(n)}_i(b(r)))|}{\sqrt{n}}\left\vert \frac{\partial^{2}\Phi_{i}^{(n)}}{\left(  \partial
b_{kh}^{(n)}\right)  ^{2}}(b(r))\right\vert \right)  \nonumber\\
&  \leq Cn^{-1/2} \left\{  \Vert f^{\prime\prime}\Vert_{\infty}+\Vert f^{\prime}%
\Vert_{\infty}\sum_{k\leq h}\left\vert \frac{\partial^{2}\Phi_{i}^{(n)}%
}{\left(  \partial b_{kh}^{(n)}\right)  ^{2}}(b(r))\right\vert \right\} ,
\end{align}
for a positive constant $C$.\par On the other hand, from p. 9 in \cite{AN} and Jensen's inequality,
\begin{align}
\sum_{k\leq h}\left\vert \frac{\partial^{2}\Phi_{i}^{(n)}}{\left(  \partial
b_{kh}^{(n)}\right)  ^{2}}(b(r))\right\vert  &  =\sum_{k\leq h}\left\vert
2\sum_{j\not =i}\frac{|u_{ik}^{(n)}(r)u^{(n)}_{jh}(r)+u_{jk}^{(n)}(r)u_{ih}^{(n)}(r)|^{2}}{\lambda
_{i}^{(n)}(r)-\lambda_{j}^{(n)}(r)}\right\vert \nonumber\label{2d2}\\
&  \leq2\sum_{k\leq h}\sum_{i\not =j}\frac{|u_{ik}^{(n)}(r)u_{jh}^{(n)}(r)+u_{jk}%
^{(n)}(r)u_{ih}^{(n)}(r)|^{2}}{|\lambda_{i}^{(n)}(r)-\lambda_{j}^{(n)}(r)|}\nonumber\\
&  \leq2\sum_{i\not =j}|\lambda_{i}^{(n)}(r)-\lambda_{j}^{(n)}(r)|^{-1},
\end{align}
We have using (\ref{2d1}) that
\begin{align}
\mathbb{E}\Bigg(  \int_{0}^{T}\Bigg(  \int_{0}^{T}\Vert D_{s}g^{i,n}%
(b(r))&\Vert^{\frac{1}{H}}ds\Bigg)  ^{pH}dr\Bigg)\notag\\
&\leq C n^{-p/2}\mathbb{E}\left(  \int_{0}^{T} r^{pH}\left\{  \Vert f^{\prime\prime}\Vert_{\infty}+\Vert f^{\prime}%
\Vert_{\infty}\sum_{k\leq h}\left\vert \frac{\partial^{2}\Phi_{i}^{(n)}%
}{\left(  \partial b_{kh}^{(n)}\right)  ^{2}}(b(r))\right\vert \right\}^p dr \right)\notag
\end{align}
Therefore, using (\ref{2d2}), and Jensen's inequality, we
obtain, for $p>1$,
\begin{align}\label{2d3}
\mathbb{E}\Bigg(  \int_{0}^{T}\Bigg(  \int_{0}^{T}&\Vert D_{s}g^{i,n}%
(b(r))\Vert^{\frac{1}{H}}ds\Bigg)  ^{pH}dr\Bigg) \notag\\
&  \leq K_{p}n^{-p/2}\int_0^Tr^{pH}\left\{  \Vert
f^{\prime\prime}\Vert_{\infty}^{p}+\Vert f^{\prime}\Vert_{\infty}^{p}%
\mathbb{E}\left(  \sum_{i\not =j}|\lambda_{i}^{(n)}(r)-\lambda_{j}%
^{(n)}(r)|^{-1}\right)  ^{p}\right\}  dr\notag\\
&  \leq K_{p}n^{-p/2}\int_0^Tr^{pH}\left\{  \Vert f^{\prime\prime}\Vert_{\infty}^{p}+\Vert
f^{\prime}\Vert_{\infty}^{p}n^{p-1}\mathbb{E}\left(  \sum_{i\not =%
j}|\lambda_{i}^{(n)}(r)-\lambda_{j}^{(n)}(r)|^{-p}\right)  \right\} dr ,
\end{align}
where $K_{p}$ is a positive constant depending on $p$.

Now using Lemma 2, we can conclude that there exists a constant $C(p)$ such
that for large enough $n\geq1$,
\[
\sum_{i\not =j}\mathbb{E}\left(|\lambda_{i}^{(n)}(r)-\lambda_{j}^{(n)}%
(r)|^{-p}\right)\leq C(p)r^{-pH}n^{p-1}.
\]
So using the above estimate in (\ref{2d3}), it is clear that there exist two constants  $K_{1}(T,p,H)$ and $K_{2}(T,p,H)$ that depend
on $p,H$ and $T$ such that
\[
\mathbb{E}\Bigg(  \int_{0}^{T}\Bigg(  \int_{0}^{T}\Vert D_{s}g^{i,n}%
(b(r))\Vert^{\frac{1}{H}}ds\Bigg)  ^{pH}dr\Bigg) 
  \leq K_{1}(T,p,H)n^{-p/2}+K_{2}(T,p,H)n^{3p/2-2}.\notag
\]
Therefore, putting all the pieces together, we have
\begin{align}
\mathbb{E}\left( \bigg|\int_{0}^{t}g^{i,n}(b(s))\delta
b(s)\bigg|^{p}\right)  \leq K_{3}(T,p,H) +K_{1}(T,p,H)n^{-p/2}+K_{2}(T,p,H)n^{3p/2-2}%
,\nonumber
\end{align}
where $K_{3}(T,p,f)$ is a positive constant that depends on $T,H$, and $p$.

In order to complete the proof, using Jensen's inequality and the
fact that $p>1$, we observe
\begin{align}
\mathbb{E}\left(  \bigg|\frac{1}{n^{3/2}}\sum_{i=1}%
^{n}\int_{0}^{t}g^{i,n}(b(s))\delta b(s)\bigg|^p\right)   &  \leq
n^{p-1}n^{-3p/2}\sum_{i=1}^{n}\mathbb{E}\left(  \sup_{0\leq t\leq T}%
\bigg|\int_{0}^{t}g^{i,n}(b(s))\delta b(s)\bigg|^{p}\right) \nonumber\\
&  \hspace{-2cm}=n^{-p/2}(K_{3}(T,p,H)+K_{1}(T,p,H)n^{-p/2}+K_{2}(T,p,H)n^{3p/2-2}%
).\nonumber
\end{align}
Hence for $\varepsilon,T>0$, and taking $p\in(1/H,2)$, we obtain that
\begin{align}
\lim_{n\rightarrow\infty}\mathbb{P}  &  \left(  \bigg|\frac{1}{n^{3/2}}\sum_{i=1}^{n}\int_{0}^{t}g^{i,n}(b(s))\delta
b(s)\bigg|>\varepsilon\right) \nonumber\\
&  \leq\lim_{n\rightarrow\infty}\frac{1}{\varepsilon^{p}}\left(
n^{-p/2}(K_{3}(T,p,H)+K_{1}(T,p,H)n^{-p/2}+K_{2}(T,p,H)n^{3p/2-2})\right)
=0.\nonumber
\end{align}
Therefore,
\[
\bigg|\frac{1}{n^{3/2}}\sum_{i=1}^{n}\int_{0}^{t}%
g^{i,n}(b(s))\delta b(s)\bigg|\rightarrow0,
\]
in probability as $n$ goes to $+\infty$.
\end{proof}

With the previous results, we are ready to identify the weak limit of the
sequence of the measure-valued processes $\{(\mu_{t}^{(n)})_{t\geq0}:n\geq1\}$
as the solution to a measure valued equation.

\begin{theorem}
\label{thfbm} 
The family of measure-valued processes $\{(\mu_{t}^{(n)})_{t\geq0}%
:n\geq1\}$ converges weakly in $C(\mathbb{R}_{+},\mathrm{Pr}(\mathbb{R}))$ to
the unique continuous probability-measure valued function satisfying, for each
$t\geq0$ $f\in C_{b}^{2}(\mathbb{R})$,
\begin{equation}
\langle\mu_{t},f\rangle=\langle \mu_0,f\rangle+H\int_{0}^{t}ds\int_{\mathbb{R}^{2}}%
\frac{f^{\prime}(x)-f^{\prime}(y)}{x-y}s^{2H-1}\mu_{s}(dx)\mu_{s}(dy).
\label{limit}%
\end{equation}
\end{theorem}

\begin{proof}
From Proposition 1, we know that the family
$\{(\mu_{t}^{(n)})_{t\geq0}:n\geq1\}$ is relatively compact. Let us now take
a subsequence $\{(\mu_{t}^{(n_{k})})_{t\geq0}:k\geq1\}$ and assume that
it converges weakly to $(\mu_{t})_{t\geq0}$. Therefore, by (\ref{Ident1}),
\begin{align}
\langle\mu_{t}^{(n_{k})},f\rangle &  -\langle\mu_{0}^{(n_{k})},f\rangle
-H\int_{0}^{t}\int_{\mathbb{R}^{2}}\frac{f^{\prime}(x)-f^{\prime}(y)}%
{x-y}s^{2H-1}\mu_{s}^{(n_{k})}(dx)\mu_{s}^{(n_{k})}(dy)ds \nonumber\label{le1}%
\\
&  =\frac{1}{n_{k}^{3/2}}\sum_{i=1}^{n_{k}}\sum_{k\leq h}\int_{0}^{t}%
f^{\prime}(\Phi^{n_{k}}_i(b(s)))\frac{\partial\Phi_{i}^{(n_{k})}}{\partial
b_{kh}^{(n_{k})}}(b(s))\delta b_{kh}(s)\nonumber\\
&  +\frac{H}{n_{k}^{2}}\sum_{i=1}^{n_{k}}\int_{0}^{t}f^{\prime\prime}%
(\Phi^{n_{k}}_i(b(s)))s^{2H-1}ds.
\end{align}
Note that for $0\leq t\leq T$, the following limit holds $\mathbb{P}$ a.s.,
\begin{equation}
\left\vert \frac{H}{n_k^{2}}\sum_{i=1}^{n_k}\int_{0}%
^{t}f^{\prime\prime}(\Phi^{n_k}_i(b(s)))s^{2H-1}ds\right\vert \leq\frac{1}%
{2n_k}T^{2H}\Vert f^{\prime\prime}\Vert_{\infty}\rightarrow0,\qquad\text{as
$k\rightarrow\infty$.} \label{le2}%
\end{equation}
Hence, using (\ref{le2}) and Lemma 3, it is clear that
\begin{align}
\lim_{k\rightarrow\infty}\bigg|\frac{1}{n_{k}^{3/2}}%
\sum_{i=1}^{n_{k}}  &  \sum_{k\leq h}\int_{0}^{t}f^{\prime}(\Phi^{n_{k}%
}_i(b(s)))\frac{\partial\Phi_{i}^{(n_{k})}}{\partial b_{kh}^{(n_{k})}%
}(b(s))\delta b_{kh}(s)\bigg|\nonumber\\
&  +\left\vert \frac{H}{n_{k}^{2}}\sum_{i=1}^{n_{k}}%
\int_{0}^{t}f^{\prime\prime}(\Phi^{n_{k}}_i(b(s)))s^{2H-1}ds\right\vert
=0,\nonumber
\end{align}
in probability, and therefore there exists a subsequence (that without loss
generality we will denote by $(n_{k})_{k\geq0}$) such that the same limit
holds $\mathbb{P}$ a.s.\newline Therefore, using (\ref{le1})
\begin{align}
\langle\mu_{t},f\rangle &  -\langle\mu_{0},f\rangle-H\int_{0}^{t}%
\int_{\mathbb{R}^{2}}\frac{f^{\prime}(x)-f^{\prime}(y)}{x-y}s^{2H-1}\mu
_{s}(dx)\mu_{s}(dy)ds\nonumber\\
&  =\lim_{k\rightarrow\infty}\langle\mu_{t}^{(n_{k})},f\rangle-\langle\mu
_{0}^{(n_{k})},f\rangle-H\int_{0}^{t}\int_{\mathbb{R}^{2}}\frac{f^{\prime
}(x)-f^{\prime}(y)}{x-y}s^{2H-1}\mu_{s}^{(n_{k})}(dx)\mu_{s}^{(n_{k}%
)}(dy)ds=0.\nonumber
\end{align}
Then we can conclude that any weak limit $(\mu_{t})_{t\geq0}$ of a subsequence
$(\mu_{t}^{(n_{k})})_{t\geq0}$ should satisfy (\ref{limit}).
\par  Applying
(\ref{limit}) to the deterministic sequence of functions
\[
f_{j}(x)=\frac{1}{x-z_{j}},\qquad\text{$z_{j}\in(\mathbb{Q}\times
\mathbb{Q})\cap\mathbb{C}^{+}$,}%
\]
and using a continuity argument, we get that the Cauchy--Stieltjes transform
$(G_{t})_{t\geq0}$ of $(\mu_{t})_{t\geq0}$ satisfies the integral equation
\begin{equation}
G_{t}(z)=-\int_{\R}\frac{\mu_0(dz)}{x-z}+H\int_{0}^{t}s^{2H-1}ds\int_{\mathbb{R}^{2}}\frac
{\mu_{s}(dx)\mu_{s}(dy)}{(x-z)(y-z)^{2}},\qquad\text{$t\geq0,z\in
\mathbb{C}^{+}$.} \label{c-st}%
\end{equation}
From (\ref{c-st}) it is easily seen that $G_{t}$ is the unique solution to the
initial value problem
\begin{equation}\label{cst}
\begin{cases}
\frac{\partial}{\partial t}G_{t}(z)=Hs^{2H-1}G_{t}(z)\frac{\partial}{\partial
z}G_{t}(z), & \qquad\text{$t>0$,}\\
G_{0}(z)=-\displaystyle \int_{\R}\frac{\mu_0(dx)}{x-z}, & \qquad\text{$z\in\mathbb{C}^{+}$}.
\end{cases}
\end{equation}
Therefore all limits of subsequences of $(\mu_{t}^{(n)})_{t\geq0}$ coincide with the family $(\mu_{t})_{t\geq0}$, with Cauchy--Stieltjes transform  given as the solution to (\ref{cst}), and thus the
sequence $\{(\mu_{t}^{(n)})_{t\geq0}:n\geq1\}$ converges weakly to $(\mu
_{t})_{t\geq0}$.
\end{proof}

\section{Convergence to a non-commutative fractional Brownian motion}

In this section, we prove that the deterministic process $(\mu_{t})_{t\geq0}$
corresponds to the law of a non-commutative fractional Brownian motion. The
intuitive idea is as follows: by the tightness of the sequence of processes
$\{(\mu_{t}^{(n)})_{t\geq0};n\geq1\}$ in the space $C(\mathbb{R}%
_{+},\mathrm{Pr}(\mathbb{R}))$, the weak limit $(\mu_{t})_{t\geq0}$ of any
subsequence $\{(\mu_{t}^{(n_{k})})_{t\geq0};k\geq1\}$ should satisfy
(\ref{limit}) for any $t\geq0$ and $f\in C_{b}^{2}$.\newline Therefore by the
uniqueness of solutions to Equation (\ref{c-st}), it is easy to check that
\[
G_{t}(z)=\frac{1}{2t^{2H}}\left(  \sqrt{z^{2}-4t^{2H}}-z\right)
,\qquad\text{$t>0$, $z\in\mathbb{C}^{+},$}%
\]
which is the Cauchy--Stieltjes transform of a semi-circle law with variance at
time $t>0$ given by $t^{2H}$, and hence the law of a non-commutative
fractional Brownian motion at time $t$.

\begin{proof}
[Proof of Theorem \ref{limit2}]
Let us recall that the sequence of processes $\{(\mu^{(n)}_{t})_{t\geq0}%
:n\geq1\}$ is tight in $C(\mathbb{R}_{+},\mathrm{Pr}(\mathbb{R}))$. This
implies that the sequence is relatively compact, in other words, there exists
a subsequence $\{(\mu^{(n_{k})}_{t})_{t\geq0}:k\geq1\}$ that converges weakly
to a process that we denote by $(\mu_{t})_{t\geq0}$ in $C(\mathbb{R}%
_{+},\mathrm{Pr}(\mathbb{R}))$.

Given the fact that the weak convergence of processes in ${C}(\mathbb{R}%
_{+},\mathrm{Pr}(\mathbb{R}))$ implies the convergence of the
finite-dimensional distributions, then for each bounded and continuous
function $g:\mathbb{R}^{m}\rightarrow\mathbb{R}$, and for each sequence of
times $0\leq t_{1}<\dots<t_{m}$, it follows that
\begin{equation}
\langle\mu_{t_{1},\dots,t_{m}}^{(n_{k})},g\rangle\overset{\mathcal{L}%
}{\rightarrow}\langle\mu_{t_{1},\dots,t_{m}},g\rangle,\qquad\text{ as
$k\rightarrow\infty$.} \label{conv}%
\end{equation}
Let us now consider $\{B^{(n_{k})}(t), t\ge0\}$ the symmetric fractional
Brownian matrix, such that the empirical measure of its eigenvalues (see
(\ref{em})) is given by $(\mu^{(n_{k})}_{t})_{t\geq0}$. 
\par First we will prove that the deterministic process $(\mu_t)_{t\geq0}$ corresponds to the law
of a semicircle process. To this end consider a set of points in time $0\leq t_{1}<\dots<t_{m}$ and $\lambda_{1}%
,\dots,\lambda_{m}\in\mathbb{R}$.  Then for any polynomial  $Q$, we have
\[%
\begin{split}
\mathbb{E}\left[\frac{1}{n_k}\mathbf{tr}\left(Q\left(\sum_{i=1}^m\lambda_{i}B^{(n_{k})}(t_{i})\right)\right)\right]= \int_{\mathbb{R}^{m}}Q\left(\sum_{i=1}^m\lambda_{i}x_{i}\right) \mu^{n_k}_{t_1,\dots,t_n}(dx_1,\dots,dx_m).
\end{split}
\]
Therefore
\begin{align}\label{semi}
\lim_{k\to\infty}\mathbb{E}\left[\frac{1}{n_k}\mathbf{tr}\left(Q\left(\sum_{i=1}^m\lambda_{i}B^{(n_{k})}(t_{i})\right)\right)\right]=\int_{\mathbb{R}^{m}}Q\left(\sum_{i=1}^m\lambda_{i}x_{i}\right)\mu
_{t_{1},t_{2},\dots,t_{m}}(dx_{1}, \dots, dx_{m}).
\end{align}

From Theorem 2.2 in \cite{vu}, we know that the random matrix
\[
X^{(n_k)}=\lambda_1B^{(n_k)}(t_1)+\dots+\lambda_mB^{(n_k)}(t_m),
\]
has a limit distribution, $\tilde{\mu}$, which is a semicircle law. Hence using (\ref{semi}) we obtain that
\[
\int_{\R}Q(x)\tilde{\mu}(dx)=\int_{\mathbb{R}^{m}}Q(\lambda_{1}x_{1}+\dots+\lambda_{m}x_{m})\mu
_{t_{1},t_{2},\dots,t_{m}}(dx_{1}, dx_{2}, \dots, dx_{m}).
\]
So if we define the function $h:\R^m\to\R$ by 
\[
h(x_1,\dots,x_m)=\sum_{i=1}^m\lambda_ix_i.
\]
Then the distribution $\mu_{t_1,\dots,t_n}\circ h^{-1}$ has a semicircle law. Therefore the process $(\mu
_{t})_{t\geq0}$ is the law of a semicircular process.
\par Now we proceed to identify the limit as the law of a non-commutative
fractional Brownian motion. So first we will prove that $(\mu_{t})_{t\geq0}$ corresponds to the law of a centered semicircular process. To this end for $t\geq 0$, we obtain, using (\ref{limit}) (with $f(x)=x$) the following
\begin{align}
\int_{\mathbb{R}}x\mu_{t}(dx) =\langle x,\delta_0\rangle=0.
\end{align}
Therefore, $(\mu_{t})_{t\geq0}$ is the law of a centered semicircular process.
\par Finally in order to conclude the proof we compute the covariance: for $t\geq s\geq0$, we obtain 
\begin{align}
\int_{\mathbb{R}^{2}}(xy)\mu_{s,t}(dx,dy)&=\lim_{k\to\infty}\int_{\mathbb{R}^{2}}(xy)\mu^{n_k}_{s,t}(dx,dy)\notag\\
&=\lim_{k\to\infty}\mathbb{E}\left[\frac{1}{n_k}\mathbf{tr}\left(B^{(n_k)}(t)B^{(n_k)}(s)\right)\right]\notag\\
&  =\frac{1}{2}\left(  t^{2H}+s^{2H}-|t-s|^{2H}\right)  .\nonumber
\end{align}
Noting that this holds for any subsequence, we can conclude that the
whole sequence $\{(\mu^{(n)}_{t})_{t\geq0}:n\geq1\}$ converges in law to the
deterministic process $(\mu_{t})_{t\geq0}$ which is characterized by being the
law of a non-commutative fractional Brownian motion.
\end{proof}

\end{document}